\documentclass[12pt]
{amsart}
\usepackage{amsfonts,amsmath,amssymb,amsthm,graphicx,afterpage,url,hyperref}

\input{xy}
\xyoption{all}

\def\lk{\mbox{lk}}

\def\id{\mathop{\fam0 id}}

\def\lk{\mathop{\fam0 lk}}

\def\im{\mathop{\fam0 im}}
\def\ker{\mathop{\fam0 ker}}

\def\R{{\mathbb R}} \def\Z{{\mathbb Z}}  \def\Q{\Bbb Q}
\let\Bbb=\mathbb
\def\inc{\mathop{\fam0 i}}

\def\cs{\phantom{}_{\inc}\#}

\newcommand{\jonly}[1]{}
\newcommand{\aronly}[1]{#1}

    \theoremstyle{theorem}
         \newtheorem{Theorem}{Theorem}[section]
         \newtheorem{Lemma}[Theorem]{Lemma}
         \newtheorem{Corollary}[Theorem]{Corollary}
         \newtheorem{Proposition}[Theorem]{Proposition}
         
    \theoremstyle{definition}
         \newtheorem{Remark}[Theorem]{Remark}

\begin{document}

\title{The band connected sum and the second Kirby move \\ for higher-dimensional links}

\author{A. Skopenkov}

\thanks{I would like to acknowledge S.~Avvakumov, T. Garaev, M. Khovanov, and the anonymous referees for useful discussions.
Most results of this paper were announced in \S2.5 of arXiv version 2 of \cite{Sk15}. 
\newline
Math. Subj. Class. Primary 57R40; secondary 55Q25.
}

\date{}

\maketitle

\abstract
Let $f:S^q\sqcup S^q\to S^m$ be an (ordered oriented) link (i.e. an embedding).

How does (the isotopy class of) the knot $S^q\to S^m$ obtained by embedded connected sum of the components of $f$ depend on $f$?

Define a link $\sigma f:S^q\sqcup S^q\to S^m$ as follows.
The first component of $\sigma f$ is the `standardly shifted' first component of $f$.
The second component of $\sigma f$ is the embedded connected sum of the components of $f$.
How does (the isotopy class of) $\sigma f$ depend on $f$?

How does (the isotopy class of) the link $S^q\sqcup S^q\to S^m$ obtained by embedded connected sum of the last two components of a link $g:S^q_1\sqcup S^q_2\sqcup S^q_3\to S^m$ depend on $g$?

We give the answers for the `first non-trivial case' $q=4k-1$ and $m=6k$.
The first answer was used by S. Avvakumov for classification of linked 3-manifolds in $S^6$.
\endabstract

\tableofcontents

\section{Introduction}\label{s:intr}


{\bf The band connected sum of a 2-component link.}

Denote $T^{0,q}:=S^0\times S^q=S^q\sqcup S^q$.
Take an (ordered oriented) link, i.e. an embedding $f:T^{0,q}\to S^m$.
(Up to isotopy this is equivalent to taking two numbered oriented $q$-spheres in $S^m$.)

Make embedded connected sum of the components of $f$ along some tube (=band) joining them.
For $m\ge q+3$ the isotopy class $\#[f]$ of this connected sum is independent of the choices of the tube, and of the link $f$ within its isotopy class $[f]$.

\begin{Remark}\label{f:sum}
A tube (for $f$) is an embedding $h:[-1,1]\times S^{q-1}\to S^m$ such that
$$h(\pm1\times S^{q-1})\subset f(\pm1\times S^q)
\quad\text{and}\quad f(T^{0,q})\cap h((-1,1)\times S^{q-1})=\emptyset.$$
Proof of the independence is obtained either using \cite[Standardization Lemma 2.1]{Sk15} or analogously to \cite[Lemma 3.2]{Sk11}, cf. \cite[Theorem 1.7]{Ha66A}, \cite[\S1]{Av17}.
The independence does not hold for $m=q+2$, e.g. for $m=q+2=3$.

For $m=q+2=3$ this operation is called {\it band-connected sum} of the components of the link.
Unlike in this paper, this operation was mostly studied for {\it split} links, for which the components are contained in disjoint cubes.
\end{Remark}


{\it How does the isotopy class $\#[f]$
depend on $[f]$?}

$\bullet$ For $2m\ge3q+4$
every two embeddings $S^q\to S^m$ are isotopic (Haefliger-Zeeman, \cite[Unknotting Spheres Theorem 2.3]{Sk16c}, \cite[Theorem 2.7.a,b]{Sk06}).

$\bullet$ We give the answer for the `first non-trivial case' $2m=3q+3$ (for
$m$ even,
see Connected Sum Theorem \ref{l:cosu}).

This answer was used for classification of linked 3-manifolds in $S^6$ (Avvakumov, \cite{Av14, Av17}).
This answer gives an alternative construction of a generator in the group of knots $S^{4k-1}\to S^{6k}$ for $k=1,2,4$ (Corollary \ref{l:cortre}.b).
Another applications are the following.

\medskip
{\bf The band connected sum of a 3-component link.}

Take a 3-component link, i.e. an embedding $g:S^q_1\sqcup S^q_2\sqcup S^q_3\to S^m$.
Make embedded connected sum of the second and the third components along some tube joining them.
We obtain a link $\#_{23}g:T^{0,q}\to S^m$.
Analogously to Remark \ref{f:sum} for $m\ge q+3$ the isotopy class $\#_{23}[g]$ of $\#_{23}g$ is independent of the choices of the tube, and of the link $g$ within its isotopy class $[g]$.

{\it How does the isotopy class $\#_{23}[g]$
depend on $[g]$?}

$\bullet$ For $2m\ge3n+4$ the linking coefficient (defined in \S\ref{s:stat}) $\lambda_{21}$ of $\#_{23}g$  equals the sum $\lambda_{21}(g)+\lambda_{31}(g)$ of the linking coefficients of $g$ ($\lambda_{21}(g)$ of the second and the first component, $\lambda_{31}(g)$ of the third and the first component); this linking coefficient $\lambda_{21}$ of $\#_{23}g$ completely defines the isotopy class of $\#_{23}g$ (Haefliger-Zeeman \cite[Theorem 4.1]{Sk16h}, \cite[Theorem 3.1]{Sk06}).


$\bullet$ We give the answer for the `first non-trivial case' $2m=3q+3$
(for $m$ even, see Theorem \ref{t:three} and Proposition \ref{r:hae}.c).

It would be interesting to obtain analogues of this and other our results for $m$ odd.

\medskip
{\bf The second Kirby move.}

Define a higher-dimensional {\it unframed second Kirby move} $\sigma$ as follows, cf.  \cite[\S3.1]{Ma80}.
The first component of the link $\sigma f:T^{0,q}\to S^m$ is the `standardly shifted' (see the details in \S\ref{s:stat}) first component of a link $f:T^{0,q}\to S^m$.
The second component of the link $\sigma f$ is the embedded connected sum of the components of $f$ along some tube joining them.
For $m\ge q+3$ the isotopy class $\sigma[f]$ of $\sigma f$ is independent of the choices of the tube, and of the  link $f$ within its isotopy class $[f]$ \cite[Lemma 3.2]{Sk11}.

{\it How does $\sigma[f]$ depend on $[f]$?}

$\bullet$ For $2m\ge3q+4$ we have $\sigma[f]=[f]$ (because the isotopy class of a link is completely defined by $\lambda_{21}$, which is preserved by $\sigma$).

$\bullet$ We give the answer for the `first non-trivial case' $2m=3q+3$ (for $m$ even,
see Theorems \ref{p:sigma} and \ref{l:cosu}, Proposition \ref{r:hae}.c).

As byproducts we

$\bullet$ show that $\sigma\ne\pm\id$ (Corollary \ref{t:nonis} and the text below), and

$\bullet$ obtain an alternative classification of links $T^{0,4k-1}\to S^{6k}$, thus proving the conjecture for $p=0$ in \cite[Remark 1.9.b]{Sk15} (Corollary \ref{t:twores}), and disproving a natural different version of this conjecture (Corollary \ref{t:tworesb}).

This is interesting because this rules out a natural inductive proof of classification of embeddings $S^p\times S^q\to S^m$ \cite[Conjecture 1.3 and Remark 1.9.c]{Sk15}.

Theorems \ref{t:three}.c and \ref{p:sigma}.b are particularly interesting because these are essentially PL results (indeed, the linking coefficients and the second Kirby move can be defined in the PL category) proved using differential topology, cf. \aronly{comment at the end of \S\ref{s:prothree}.}\jonly{\cite[comment after Theorem 2.6]{Sk24}.}

\section{Statements of main results}\label{s:stat}

Theorems \ref{l:cosu}, \ref{t:three} and  \ref{p:sigma} completely describe the respective operations in terms of respective complete isotopy invariants of the involved groups.

We work in the smooth category which we omit from the notation.
Analogues of Theorem \ref{t:three}.abc, \ref{p:sigma}, and Remark \ref{r:doubling} for the PL category are correct, cf. \cite[\S6]{Sk16h}.

For a manifold $N$ denote by $E^m(N)$ the set of embeddings $N\to\R^m$ up to isotopy.
By $[\cdot]$ we denote the isotopy class of an embedding or the homotopy class of a map.

We assume that $m\ge q+3$, unless indicated otherwise.

The sum operations on $E^m(T^{0,q})$ and on $E^m(S^q)$ are `embedded connected sums of two embeddings whose
images are contained in disjoint cubes'.
See
detailed definition of abelian group structures on these sets in \cite{Ha66A, Ha66C}.

Identify $E^{6k}(S^{4k-1})$ with $\Z$ by the isomorphism of \cite{Ha62k, Ha66A}.

\bigskip
{\bf The band connected sum of a 2-component link.}

Below we define and use the following diagram of groups and homomorphisms.
$$\xymatrix{
E^m(S^q) \ar[r]_{s_\pm} & E^m(T^{0,q}) \ar@(ul,ur)[l]_{r_\pm} \ar@(dl,dr)[l]^{\#} \ar[r]_{\lambda_\pm} &
\pi_q(S^{m-q-1}) \ar@(ul,ur)[l]_{\zeta}}.$$
The map $\#$ defined in \S\ref{s:intr} is clearly a homomorphism.

\smallskip
{\bf Definition of $r_\pm$ and $\lambda_\pm$.}
Let $r_\pm$ be `the knotting class of the component', i.e. $r_\pm$ is induced by the inclusion $\pm1\times S^q\subset T^{0,q}$, where the orientation on $\pm1\times S^q$ corresponds to the standard orientation on $S^q$.

Let $\lambda_\pm$ be the linking coefficient, i.e. the homotopy class of $f|_{\pm1\times S^q}$
in the complement to the other component, see
detailed definition in \cite[\S3]{Sk16h}, \cite[\S3]{Sk06}.


\smallskip
Let
$$H:\pi_{4k-1}(S^{2k})\to\Z$$
be the Hopf invariant (defined to be the linking number of preimages of two regular points under a smooth or a PL approximation of a map $S^{4k-1}\to S^{2k}$).

Descriptions of $E^{6k}(T^{0,4k-1})$ in terms of the invariants $\lambda_+,\lambda_-,r_+,r_-$ or $H\lambda_+,\lambda_-,r_+,r_-$ is given by the Haefliger Theorem \ref{t:hae} and Proposition \ref{r:hae}.

\begin{Theorem}[Connected Sum; proved in \S\ref{s:procosu}]\label{l:cosu}
For $q=4k-1$ and $m=6k$ we have
$$\#=r_++r_-\pm\dfrac{H\lambda_++H\lambda_-}2.$$
\end{Theorem}

The sign in this formula (and in Corollary \ref{l:cortre}.a, Theorem \ref{t:three}) could depend on $k$.

The integers $H\lambda_+$ and $H\lambda_-$ have the same parity by Proposition \ref{r:hae}.b.

Recall that $D^n = \{(x_1,\ldots,x_n)\ :\ x_1^2+\ldots+x_n^2\le1\}$.
The natural normal framing on the inclusion $S^q\to S^m$ defines {\bf the standard embedding}
$$\inc = \inc\phantom{}_{m,q} : D^{m-q}\times S^q\to S^m.$$
For a map $x:S^q\to S^{m-q-1}$ let us define an embedding $\overline\zeta_x:T^{0,q}\to S^m$.
Speaking informally, we set $\overline\zeta_x(-1\times S^q) := \inc(0\times S^q)$, and
take $\overline\zeta_x(1\times S^q)$ to be the $\inc$-image of the graph of $x$ in $S^{m-q-1}\times S^q$.
Rigorously speaking, we need to define a map not only its image.
Define
$$\overline\zeta\phantom{}_x(t,a) := \inc\left(\frac{1+t}2x(a),a\right).$$
Define {\bf the Zeeman map} $\zeta$ by $\zeta[x]:=[\overline\zeta_x]$.

Clearly, $\zeta$ is well-defined, is a homomorphism, and $\lambda_+\zeta=\id\pi_q(S^{m-q-1})$,
see \cite[Theorem 10.1]{Ha66C}, \cite[Definition of Ze before Lemma 3.4]{Sk11}, \cite[Remarks 3.2.ac]{Sk16h}.

\begin{Corollary}\label{l:cortre} (a) We have $\#\zeta=\pm H$ on $\pi_{4k-1}(S^{2k})$.

(b) For any $k=1,2,4$ let $\eta\in\pi_{4k-1}(S^{2k})$ be the homotopy class of the Hopf map.
The embedded connected sum $\#\zeta\eta$ of the components of $\zeta\eta$ is a generator of $E^{6k}(S^{4k-1})\cong\Z$.
\end{Corollary}

Part (b) follows from (a) because $H\eta=1$.
Part (a) is proved in \S\ref{s:procosu}.

\bigskip
{\bf The band connected sum of a 3-component link.}

Take an embedding $g:S^q_1\sqcup S^q_2\sqcup S^q_3\to S^m$.

Let $r_i=r_i(g)\in E^m(S^q)$, $i\in[3]$, be the isotopy classes of the restrictions of $g$ to the components.
Let
$$\lambda_{ij}=\lambda_{ij}(g)\in\pi_q(S^{m-q-1}),\quad (i,j)\in[3]^2,\quad i\ne j,$$
be the pairwise linking coefficients of the components ($\lambda_{ij}=\lambda_+(g|_{S^q_i\sqcup S^q_j})$ is the class of the $i$-th component in the complement of the $j$-th component).
For $m=6k$ and $q=4k-1$ denote
$$h_{ij}=h_{ij}(g)=\frac12(H\lambda_{ij}+H\lambda_{ji}).$$
Let $\lambda^1=\lambda^1(g)\in \pi_q(S^{m-q-1}\vee S^{m-q-1})$ be homotopy class of $g|_{S^q_1}$ in $S^m-g(S^q_2\sqcup S^q_3)$, see
detailed definition in \cite[\S4]{Ha62l}, \cite{HS64}, \cite[proof of Theorem 9.4]{Ha66C}.
For $3m\ge4q+6$ let the \textbf{triple linking coefficient} $\mu=\mu(g)$ be the image of $\lambda^1$ under the composition
$$\pi_q(S^{m-q-1}_2\vee S^{m-q-1}_3) \to \frac{\pi_q(S^{m-q-1}_2\vee S^{m-q-1}_3)}{i_{2*}\pi_q(S^{m-q-1}_2)\oplus i_{3*}\pi_q(S^{m-q-1}_3)} \to \pi_q(S^{2m-2q-3})\overset{\Sigma^\infty}\to\pi^S_{3q-2m+3}$$
of the projection from the Hilton
theorem and the stable suspension.
Recall that $\pi^S_n:=\pi_{n+s}(S^n)$ for $s\ge n+2$ is the \emph{stable homotopy group of spheres}; we identify $\pi^S_0$ and $\Z$ by the degree isomorphism.

\smallskip
{\bf Comment.} Since $3m\ge4q+6$, we have $q<3(m-q)-5$, so the second map is an isomorphism;
since $3m\ge4q+6$, we have $q\le2(2m-2q-3)$, so the stable suspension $\Sigma^\infty$ is an isomorphism; $\mu$ was denoted by $\lambda^1_{23}$ in \cite[\S4]{Ha62l}.

\smallskip
See alternative definitions of $\mu$ in Lemma 4.1, \cite{Ma90, Ko91}, \cite[\S9.1]{Mo08}.

\begin{Theorem}\label{t:three} For $q=4k-1$ and $m=6k$ we have

(a) $\lambda_-\#_{23} = \lambda_{21}+\lambda_{31}$ and $r_+\#_{23} = r_1$.

(b) $r_-\#_{23} = r_2+r_3\pm h_{23}$.


(c) $H\lambda_+\#_{23} = 2\mu+H\lambda_{12}+H\lambda_{13}$.

(d) $\#\#_{23} = r_1+r_2+r_3\pm(\mu+h_{12}+h_{23}+h_{31})$.
\end{Theorem}


Here (a) is obvious (and holds whenever $m-q\ge3$), (b) holds by Connected Sum Theorem \ref{l:cosu}, while (c,d) are non-trivial (they are proved in \S\ref{s:prothree} using interpretation of linking coefficients via Pontryagin construction, see Lemmas \ref{l:lkpon} and \ref{l:mupt}).
Parts (c) and (d) are equivalent (but they are proved together, not deduced one from the other).

\bigskip
{\bf The second Kirby move.}

Take the inclusion of $D^1=[-1,1]$ in $D^{m-q}$ given by $x\mapsto(x,0,\ldots,0)$.
Represent an element of $\ker r_+$ by an embedding
$$f:T^{0,q}\to S^m\quad\text{such that}\quad f=\inc\quad\text{on}\quad 1\times S^q\quad\text{and}\quad f(-1\times S^q)\cap\inc(D^1\times S^q)=\emptyset.$$
Define embedding $g:T^{0,q}\to S^m$ as $g(1,x)=\inc(0,x)$, and as the embedded connected sum
$f|_{1\times S^q}\#f|_{-1\times S^q}$ of the components of $f$, with parallel orientations, on $-1\times S^q$.
Define {\bf the unframed second Kirby move}
$$\sigma:\ker r_+ \to \ker r_+ \quad\text{by}\quad \sigma[f]:=[g].$$
The map $\sigma$ is well-defined and is a homomorphism \cite[Lemmas 3.1--3.3]{Sk11}.


\begin{Remark}\label{r:p=0}
The map $\sigma$ is an isomorphism (this fact is not used elsewhere in this paper).
This follows from \cite[Theorem 1.6]{Sk11} because the maps $i^*$ in the exact sequence are surjections.
A simple direct proof is as follows.
Define $\sigma':\ker r_+\to\ker r_+$ analogously to $\sigma$ but taking embedded connected sum of $f|_{1\times S^q}$ with reversed orientation and $f|_{-1\times S^q}$.
The second component of a representative of $\sigma'\sigma[f]$ is obtained from the second component
of $f$ by adding the two first components with different orientations, which cancel outside the shifted first component.
Therefore $\sigma'\sigma=\id$.
Analogously $\sigma\sigma'=\id$.
Thus $\sigma'$ is the inverse of $\sigma$.
\end{Remark}


\begin{Theorem}\label{p:sigma} (a) We have $r_-\sigma=\#$ and $\lambda_-\sigma=\lambda_-$.

(b) For $q=4k-1$ and $m=6k$ we have $H\lambda_+\sigma=H\lambda_++2H\lambda_-$.
\end{Theorem}

Here

$\bullet$ the formula for $r_-\sigma$ is obvious;

$\bullet$ the formula for $\lambda_-\sigma$ follows since in the definition of $\sigma$ the restrictions of $g$ and $f$ to the second component $-1\times S^q$ are homotopic as maps to
$$S^m-g(1\times S^q) = S^m-\inc(1\times S^q) \sim S^m-\inc(D^1\times S^q) \sim S^m-\inc(0\times S^q) = S^m-f(1\times S^q);$$

$\bullet$ part (b) is non-trivial and is proved in \S\ref{s:prokirby} (together with the following Corollaries \ref{t:nonis}-\ref{t:tworesb}).

\smallskip
{\bf Definition of $s_\pm$.}
For an embedding $g:S^q\to S^m$ let $s_\pm(g)$ be any link whose components are contained in disjoint balls,
whose restriction to $\pm1\times S^q$ is $g$ and whose restriction to the other component is the standard embedding.
Define $s_\pm[g]:=[s_\pm(g)]$.

\begin{Corollary}\label{t:nonis}
If $q=4k-1$ and $m=6k$, then $\im(\sigma\zeta-\zeta)\not\subset\im(s_+\oplus s_-)$.

In other words, there is $x\in\pi_{4k-1}(S^{2k})$ such that no link representing $\sigma\zeta x-\zeta x$ is
piecewise-smoothly isotopic to the standard link.
\end{Corollary}

The analogue of Corollary \ref{t:nonis} for $\sigma\zeta+\zeta$ instead of $\sigma\zeta-\zeta$ follows just because $\lambda_-\sigma=\lambda_-$ by Theorem \ref{p:sigma}.a.
Corollary \ref{t:nonis} is non-trivial because of $\lambda_-\sigma=\lambda_-$ (Theorem \ref{p:sigma}.a) and Lemma \ref{l:sigmala}.

\begin{Corollary}\label{t:twores}
If $q=4k-1$ and $m=6k$, then the following map is an isomorphism:
$$r\oplus\sigma|_K\oplus s_-:E^m(D^1\times S^q)\oplus K\oplus E^m(S^q)\to E^m(T^{0,q}).$$
Here $r$ is the restriction map, $K:=\ker(\lambda_+\oplus r_+\oplus r_-)$, and the sum operation on
$E^m(D^1\times S^q)$ is `embedded connected sum of $q$-spheres together with normal vector fields', see detailed construction in \cite[\S2.1]{Sk15}.
\end{Corollary}

{\bf Definition of $\psi_\pm$.}
Let $\psi_\pm$ be the `change of the orientation of $\pm1\times S^q$' self-map of $E^m(T^{0,q})$.

\smallskip
{\bf Comment.}
The map $\psi_+$ is described for $q=4k-1$ and $m=6k$ by Symmetry Lemma \ref{l:symkno}.b and Proposition \ref{r:hae}.c.
A description of $\psi_-$ is analogous.
\aronly{The map $\sigma'$ of the Comment before Theorem \ref{p:sigma} equals $\psi_+\sigma\psi_+$.}

\smallskip
The map $\psi_-r$ can be considered as the restriction map corresponding to the `boundary inclusion' $T^{0,q}\to D^1\times S^q$, for which the orientations on the components of $T^{0,q}$ are the boundary orientations
(as opposed to the ordinary `product inclusion' used for $r$).

\smallskip
{\bf Comment.}
Observe that $\#\psi_-r=0$ for any $m,q$ where $\#$ is defined, because a representative of $\#\psi_-r$ spans a ball in $\R^m$.
This is different from $\#r\ne0$ (which follows because Corollary \ref{l:cortre}.a implies that $\#\zeta\ne0$, and $\zeta=r\tau$ for certain map $\tau$, cf. the text before Lemma \ref{l:sigma}).

\begin{Corollary}\label{t:tworesb}
If $q=4k-1$ and $m=6k$, then the following map is not surjective:
$$\psi_-r\oplus\sigma|_K\oplus s_-:E^m(D^1\times S^q)\oplus K\oplus E^m(S^q)\to E^m(T^{0,q}).$$
\end{Corollary}

Concerning low-dimensional version of $\#$ and $\sigma$ see \aronly{Remark \ref{r:nonis}.}
\jonly{\cite[Remark 5.6]{Sk24}.}

\section{Proof of Theorem \ref{l:cosu}}\label{s:procosu}


Our proof of Connected Sum Theorem \ref{l:cosu} (and of Corollary \ref{l:cortre}, Theorem \ref{p:sigma}.b) is not by directly using definition of the isomorphism $E^{6k}(S^{4k-1})\to\Z$ (or of the map $H\lambda_+:E^{6k}(T^{0,4k-1})\to\Z$); we use known classification of $E^{6k}(T^{0,4k-1})$.
It would be interesting to obtain a direct proof.

First we recall known results used in our proof.

Let $\Sigma:\pi_{4k-1}(S^{2k})\to\pi_{4k}(S^{2k+1})$ the suspension homomorphism.

\begin{Theorem}[Haefliger]\label{t:hae} The following map is a monomorphism:
$$\lambda_+\oplus \lambda_-\oplus r_+\oplus r_- :
E^{6k}(T^{0,4k-1}) \to \pi_{4k-1}(S^{2k})\oplus\pi_{4k-1}(S^{2k})\oplus\Z\oplus\Z.$$
Its image is the set of quadruples $(a_+,a_-,b_+,b_-)$ such that $\Sigma(a_++a_-)=0$.
\end{Theorem}

See \cite[Theorem in \S6]{Ha62l}, \cite[Theorem 2.4 and appendix]{Ha66C}, \cite[Remark 7.2.b]{Sk16h}.

\begin{Theorem}[{\cite[Lecture 6, (7)]{Po85}}]\label{t:whit} The kernel $\ker H$ is finite.
The kernel $\ker\Sigma$ is generated by an element $w$ such that $Hw=2$.
\end{Theorem}

\begin{Proposition}\label{r:hae}
(a) The following map has a finite kernel:
$$H\lambda_+\oplus H\lambda_- \oplus r_+ \oplus r_- : E^{6k}(T^{0,4k-1}) \to \Z^4.$$

(b) The image of this map is $\{(a,b,c,d)\ :\ a\equiv b\mod2\}$ for $k=1,3,7$, and is
$\{(a,b,c,d)\ :\ a\equiv b\equiv0\mod2\}$   otherwise.

(c) The following map is a monomorphism
$$H\lambda_+\oplus \lambda_- \oplus r_+ \oplus r_- : E^{6k}(T^{0,4k-1}) \to  \Z\oplus\pi_{4k-1}(S^{2k})\oplus\Z\oplus\Z.$$

(d) The image of this map is $\{(a,b,c,d)\ :\ a\equiv Hb\mod2\}$ for $k=1,3,7$, and is
$\{(a,b,c,d)\ :\ a\equiv Hb\equiv0\mod2\}$   otherwise.
\end{Proposition}

\begin{proof}
Part (a) holds by (c) and Theorem \ref{t:whit}.

Part (b) holds by (d).

Part (c) follows because $H\lambda_+x = \lambda_-x = 0$ implies that $\lambda_+x = 0$.
In the rest of this paragraph we prove the latter implication.
Since $\lambda_-x = 0$, by Theorems \ref{t:hae} and \ref{t:whit} $\lambda_+x = sw$ for some integer $s$.
Since $0 = H\lambda_+x = sHw = 2s$, we have $s=0$.

Part (d) holds by Theorems \ref{t:hae} and \ref{t:whit}, because $H$ is surjective for $k=1,3,7$, and  $H\pi_{4k-1}(S^{2k})=2\Z$ otherwise.
\end{proof}

\begin{proof}[Proof of Connected Sum Theorem \ref{l:cosu}]
Clearly, $\#,r_+,r_-,\lambda_+,\lambda_-$ are homomorphisms.
The group $E^{6k}(T^{0,4k-1})$ is generated by (isotopy classes of) links whose components are contained in disjoint smooth balls, and by $K_0:=\ker(r_+\oplus r_-)$.
We have $\#=r_++r_-$ for the former links.
Hence it suffices to prove the theorem for links in $K_0$.

By Proposition \ref{r:hae}.a the map $H\lambda_+\oplus H\lambda_-:K_0\to\Z^2$ has finite kernel.
Since any homomorphism from a finite group to $\Z$ is zero, this kernel goes to 0 under the map $\#$.
Hence $\#|_{K_0}=n\circ(H\lambda_+\oplus H\lambda_-)$ for some homomorphism $n:\im(H\lambda_+\oplus H\lambda_-)\to\Z$.
So $\#|_{K_0}=n_+H\lambda_++n_-H\lambda_-$ for some $n_\pm=n_{\pm,k}\in\Q.$
Analogously to the commutativity of summation on $E^m(S^q)$ \cite[\S1.4]{Ha66A}, $\#$ is invariant under exchange of the components.
Hence $n_+=n_-$.
So by the following Whitehead Link Lemma \ref{l:whli} $n_+=\pm1/2$.
\end{proof}


\begin{Lemma}[Whitehead Link]\label{l:whli}
For any $l\ge2$ there is an embedding $\omega:T^{0,2l-1}\to S^{3l}$ such that
$$r_+\omega=r_-\omega=0,\quad\lambda_-\omega=0,\quad\text{and, for $l$ even, }\quad H\lambda_+\omega=\pm2,\quad \#\omega=1.$$
\end{Lemma}

\begin{proof}[Proof of Lemma \ref{l:whli} except $H\lambda_+\omega=\pm2$] (The following construction of Borromean rings and their spanning disks is known \cite[\S4]{Ha62k},
and the proof modulo this construction is not hard.)

Recall that isotopy classes of embeddings $S^q\to S^n$ are in 1--1 correspondence with $h$-cobordism classes of oriented submanifolds of $S^n$ diffeomorphic to $S^q$ for $n\ge5$, cf. \cite[1.8]{Ha66A}.

Denote coordinates in $\R^{3l}\subset S^{3l}$ by $(x,y,z)=(x_1,\ldots,x_l,y_1,\ldots,y_l,z_1,\ldots,z_l)$.
The Borromean rings is the embedding whose image is disjoint union $S_1\sqcup S_2\sqcup S_3\to\R^{3l}$
of the three $(2l-1)$-spheres given by the following three systems of equations
$$\begin{cases} x=0\\ |y|^2+2|z|^2=1 \end{cases},\qquad
\begin{cases} y=0\\ |z|^2+2|x|^2=1\end{cases}\qquad\text{and}\qquad
\begin{cases} z=0\\ |x|^2+2|y|^2=1\end{cases}.$$
The embedding (up to isotopy) is defined by taking the orientations on the components as described in \cite[\S4]{Ha62k}.

Let $\omega:T^{0,2l-1}\to\R^{3l}$ be an embedding such that $\omega_{1\times S^{2l-1}}$ is (defined up to isotopy as) oriented $S_1$, and $\omega|_{-1\times S^{2l-1}}$ is embedded connected sum of oriented $S_2$ and $S_3$ along some tube $\tau\cong\partial D^{2l-1}\times I$ joining $S_2$ and $S_3$.

For $l$ even $\#\omega=1$ by \cite[\S4]{Ha62k}.

In this paragraph we prove that \emph{for each $i,j\in\{1,2,3\}$, $i\ne j$, there are disjoint $2l$-disks $D_{ij},D_{ji}\subset\R^{3l}$ bounded by $S_i$ and $S_j$, respectively}.
By symmetry, it suffices to prove this for $i=2$, $j=3$.
Take $2l$-disks $D_{23},D_{32}\subset\R^{3l}$ given by the equations
$$\begin{cases} y=0 \\ |z|^2+2|x|^2\le1 \end{cases}\qquad\text{and}\quad
\begin{cases} z_1\ge0 \\ z_2=\ldots=z_l=0 \\ |x|^2+2|y|^2+\frac12|z|^2=1\end{cases}.$$
These disks are bounded by $S_2$ and $S_3$, respectively.
On the intersection $D_{23}\cap D_{32}$ we have $2=2|x|^2+|z|^2\le1$, hence $D_{23}\cap D_{32}=\emptyset$.

Clearly, $r_+\omega=0$.
Since the spheres $S_2$ and $S_3$ bound disjoint embedded $2l$-disks $D_{23}$ and $D_{32}$, we have $r_-\omega=0$.

Take oriented embedded boundary connected sum of $D_{21}$ and $D_{31}$ by a half-tube $D^{2l-1}\times I$
disjoint from $S_1$, and such that $\partial D^{2l-1}\times I=\tau$.
We obtain a self-intersecting $2l$-disk bounded by $\omega(-1\times S^{2l-1})$ and disjoint from $S_1$.
Then $\lambda_-\omega=0$.
(An informal explanation for $\lambda_-\omega=0$ is that by making self-intersection of the last two of the Borromean rings, we can drag them apart from the first ring.)
\end{proof}

We identify by the Pontryagin isomorphism \cite[\S18.5]{Pr06} the group $\pi_q(S^n)$ and the set of framed cobordism classes of framed $(q-n)$-submanifolds of $S^q$.

\begin{Lemma}\label{l:lkpon}
(a) Let $f:T^{0,q}\to S^m$ be a link such that $r_+f=r_-f=0$.
Then $\lambda_+f$
is equal to the framed intersection of
$f(1\times S^q)$ and a general position arbitrarily framed embedded $(q+1)$-disk spanned by $f(-1\times S^q)$.


(b) If framed $(2k-1)$-submanifolds $\alpha,\beta$ of $S^{4k-1}$ are disjoint, then
\linebreak
$H(\alpha\sqcup\beta)=H\alpha+H\beta+2\lk_{S^{4k-1}}(\alpha,\beta)$.
\end{Lemma}

Part (a) is proved analogously to the particular case \cite[Lemma 4.1]{Av17}.
(Although the statement of \cite[Lemma 4.1]{Av17} involved Hopf invariant, the proof calculated $\lambda_+f$ not $H\lambda_+f$.)
Part (b) follows because $H\alpha$ is the linking number of $\alpha$ and the shift of $\alpha$ along the first vectors of the framing.
Although part (b) is not published, and part (a) is not published before \cite{Av17}, both parts are presumably folklore results known before \cite{Av17}.

\begin{proof}[Proof of Lemma \ref{l:whli}: proof that $H\lambda_+\omega=\pm2$ for $l$ even]
(This was stated without proof in \cite[end of \S6]{Ha62l}, and the proof presented below is not hard.)

Take oriented embedded boundary connected sum of $D_{23}$ and $D_{32}$ by a half-tube $D^{2l-1}\times I$
disjoint from $S_1$, and such that $\partial D^{2l-1}\times I=\tau$.
Since $D_{23}\cap D_{32}=\emptyset$, we obtain an embedded $2l$-disk bounded by $\omega(-1\times S^{2l-1})$.
Its intersection with $S_1$ is
$$\Sigma_1:=(D_{23}\cap S_1)\sqcup(D_{32}\cap S_1).$$
The intersection $D_{23}\cap S_1$ is transversal, and is the $(l-1)$-sphere given by $x=y=0$, $|z|^2=1/2$.
The intersection $D_{32}\cap S_1$ is transversal, and is the $(l-1)$-sphere given by
$$x=z_2=\ldots=z_l=0,\quad z_1=\sqrt{2/7},\quad |y|^2=3/7.$$
Take any normal framings on $D_{23}, D_{32}$ and $S_1$.
Consider $D_{23}\cap S_1$, $D_{32}\cap S_1$, and $\Sigma_1$ as framed intersections.
Take the orientations on $D_{23}\cap S_1$, $D_{32}\cap S_1$ corresponding to the framings.
Then
$$H\lambda_+\omega \overset{(1)}= H\Sigma_1 \overset{(2)}= 2\lk\phantom{}_{S_1}(D_{23}\cap S_1,D_{32}\cap S_1)
\overset{(3)}= \pm2.$$
Here equality (1) holds by Lemma \ref{l:lkpon}.a applied to $f=\omega$ so that $f(1\times S^q)=S_1$ and the framed intersection is $\Sigma_1$.

Let us prove equality (2).
The spheres $S_1$ and $S_2$ bound disjoint disks $D_{12}$ and $D_{21}$.
So  the framed intersection $D_{23}\cap S_1$ is framed cobordant to zero.
Hence
$H(D_{23}\cap S_1)=0$.
Analogously $H(D_{32}\cap S_1)=0$.
Now Lemma \ref{l:lkpon}.b implies equality (2).


Let us prove equality (3).
The $(l-1)$-sphere $D_{23}\cap S_1$ bounds in $S_1$ the $l$-disk given by
$$x=y_2=\ldots=y_l=0,\quad y_1\ge0,\quad y_1^2+2|z|^2=1.$$
The intersection of this $l$-disk and $D_{32}\cap S_1$ is the only point
$$x=y_2=\ldots=y_l=z_2=\ldots=z_l=0,\quad y_1=\sqrt{3/7},\quad z_1=\sqrt{2/7}.$$
This is a transversal intersection point.
This implies equality (3).
\end{proof}

\aronly{

\begin{proof}[Alternative proof of $H\lambda_-\omega=\pm2$ for $l$ even, using $\#\omega=1$]
By the Haefliger Theorem \ref{t:hae} $\Sigma\lambda_-\omega=-\Sigma\lambda_+\omega=0$.
Then by Theorem \ref{t:whit} $H\lambda_-\omega$ is an even number, say $2s$.
Since $\lambda_+\omega=0$ and $r_\pm\omega=0$, by Theorems \ref{t:hae} and \ref{t:whit} we have that   
$\omega$ is divisible by $s$.
Then $\#\omega$ is divisible by $s$.
Since $\#\omega=1$, we obtain $|s|=1$.
\end{proof}

\begin{proof}[Sketch of an alternative proof of Connected Sum Theorem  \ref{l:cosu} for $k=1$]
Take a representative $f:T^{0,3}\to S^6$ of an isotopy class from $E^6(T^{0,3})$.
Denote by $g$ a representative of $\#[f]$.
The formula follows by \cite[Theorem 4]{Wa66} because `the homology class of handle' $g_0\in H_4(M_g)$ `goes to' $f_++f_-$, so $[g]=(f_++f_-)^3/6$.
\end{proof}

}

\begin{proof}[Proof of Corollary \ref{l:cortre}.a]
Part (a) follows by Connected Sum Theorem \ref{l:cosu} and Lemma \ref{l:sym} below since $r_\pm\zeta=0$ and $\lambda_+\zeta=\id\pi_q(S^{m-q-1})$.
\end{proof}

\begin{Lemma}\label{l:sym} We have $H\lambda_-\zeta=H$ on $\pi_{4k-1}(S^{2k})$.
\end{Lemma}

Let $\iota_n\in\pi_n(S^n)$ be the homotopy class of the identity map.

\begin{proof}[Proof of Lemma \ref{l:sym}]
For a map $y:S^q\to S^{m-q-1}$ the link obtained from $\overline\zeta_y$ by exchange of components is isotopic to $\overline\zeta_{S\circ y}$, where $S$ is the symmetry of $S^{m-q-1}$ w.r.t the origin.
Then $\lambda_-\zeta[y]=((-1)^{m-q}\iota_{m-q-1})\circ[y]$.
So $H\lambda_-\zeta[y] = H((-\iota_{2k})\circ[y]) = H[y]$ by Lemma \ref{l:circ} below.
\end{proof}

\begin{Lemma}[known]\label{l:circ} For any $x\in\pi_{4k-1}(S^{2k})$ we have $H((-\iota_{2k})\circ x)=Hx$.
\end{Lemma}

\begin{proof}
Change of the orientation of $S^{2k}$ changes the orientations on preimages of regular points under a map $S^{4k-1}\to S^{2k}$.
Change of the orientation of both components preserves the linking number.
Hence change of the orientation of $S^{2k}$ preserves the Hopf invariant.
\end{proof}


\section{Proof of Theorems  \ref{t:three}.c,d}\label{s:prothree}

\begin{Lemma}\label{l:mupt} Let

$\bullet$ $g:S_1^{2l-1}\sqcup S_2^{2l-1}\sqcup S_3^{2l-1}\to S^{3l}$ be an embedding such that $\lambda_{23}(g)=\lambda_{32}(g)=0$, and

$\bullet$ $D_2,D_3\subset S^{3l}$ be disjoint oriented embedded $2l$-disks in general position to $g_1:=g|_{S^{2l-1}_1}$, and such that $g(S^{2l-1}_j)=\partial D_j$ for each $j=2,3$.

Then for $j=2,3$ the oriented preimage $g_1^{-1}D_j$ is a closed oriented $(l-1)$-submanifold of $S^{2l-1}_1$ missing $g_1^{-1}D_{5-j}$, and $\mu(g)=\lk_{S^{2l-1}_1}(g_1^{-1}D_2,g_1^{-1}D_3)$.
\end{Lemma}

This holds by the well-known `linking number' definition of the Hopf-Whitehead invariant
$\pi_{2l-1}(S^l\vee S^l)\to\Z$, see e.g. \cite[\S2, Sketch of a proof of (b1)]{Sk20e}.

\begin{proof}[Proof of Theorems \ref{t:three}.c,d]
Clearly, $\#_{23},r_i,\lambda_{ij}$ are homomorphisms.
Also $\mu$ is a homomorphism.
We have $\lambda_+\#_{23}=0$ and $\#\#_{23}=r_1+r_2+r_3$ for links whose components are contained in pairwise disjoint smooth balls.
Hence (analogously to the proof of Theorem \ref{l:cosu}) it suffices to prove (c,d) for links in $K_0:=\ker(r_1\oplus r_2\oplus r_3)$.

If $\lambda_{23}g=\lambda_{32}g=0$, then by Proposition \ref{r:hae}.c the spheres $g(S_2^{4k-1})$ and $g(S_3^{4k-1})$ bound disjoint embedded $4k$-disks.
Then (c) follows by Lemmas \ref{l:lkpon}.ab and \ref{l:mupt}.

The sum
$$\mu\oplus\sum\limits_{(i,j)\in[3]^2,\ i\ne j}\lambda_{ij} : K_0 \to \Z\oplus\pi_{4k-1}(S^{2k})^6$$
has a finite kernel by \cite[Theorem 9.4]{Ha66C} and \cite[Lemma 1.3]{CFS}, see also \cite[Theorem 9.3.b]{Sk16h}.
(Note that this sum is a monomorphism for $k>1$ by \cite[Theorem in \S6]{Ha62l}, \cite[Theorem 9.4]{Ha66C}, see also \cite[Remarks 8.2ab and Theorem 8.3]{Sk16h}.)
Since any homomorphism from a finite group to $\Z$ is zero, (analogously to the proof of Theorem \ref{l:cosu})
by Proposition \ref{r:hae}.a and the case $\lambda_{23}g=\lambda_{32}g=0$ of (c) we have
$$H\lambda_+\#_{23}|_{K_0} = 2\mu+l_{23}H\lambda_{23}+l_{32}H\lambda_{32}+H\lambda_{12}+H\lambda_{13}\quad\text{for some}\quad l_{23},l_{32}\in\Q.$$
Analogously to the commutativity of summation on $E^m(S^q)$ \cite[\S1.4]{Ha66A}, $\#_{23}$ is invariant under exchange of the second and the third components.
By \cite[Theorem in p. 259]{HS64} $\mu$ is invariant under any permutation of all the three components.\aronly{\footnote{Clearly,
there is a typo in \cite[\S6, Theorem, (2)]{Ha62l} because the sign could not depend on the numbering.
Clearly, there is a typo in \cite[Theorem in p. 259]{HS64}: $p_1,p_2,p_3$ should be $i,j,k$, respectively.
Hence \cite[\S6, Theorem, (2)]{Ha62l} should read as $\lambda^i_{jk}=\lambda^j_{ik}=\lambda^k_{ij}$.}}
Hence $l_{23}=l_{32}$.

Then by Connected Sum Theorem \ref{l:cosu}
$$\#\#_{23}|_{K_0} =
\pm h_{23}\pm\frac12(H\lambda_{21}+H\lambda_{31}+2\mu+2l_{23}h_{23}+H\lambda_{12}+H\lambda_{13}).$$
Under any permutation of all the three components both $\#\#_{23}$ and $\mu$ remain the same.
Hence $l_{23}=0$.
\end{proof}

\aronly{It is not clear how to prove Theorems \ref{t:three}.c and \ref{p:sigma}.b directly in the PL category because its is not clear how to calculate directly $H\lambda_+\#_{23}$ and $H\lambda_+\sigma$; the calculations in \S\S \ref{s:procosu},\ref{s:prothree},\ref{s:prokirby} use Connected Sum Theorem \ref{l:cosu}, whose PL version is an `empty' result because any two PL embeddings $S^{4k-1}\to\R^{6k}$ are PL isotopic.}

\section{Proof of Theorem \ref{p:sigma}.b and Corollaries \ref{t:nonis}-\ref{t:tworesb}}\label{s:prokirby}

Part (a) of the following lemma should be compared to the analogous result \cite[\S3, Symmetry Remark]{Sk05} on embeddings $S^4\to S^7$, where situation is `the opposite'.

\begin{Lemma}[Symmetry]\label{l:symkno} (a) For any embedding $g:S^{4k-1}\to S^{6k}$ the composition with the
reflection-symmetry of $S^{6k}$ is isotopic to $g$.

Or, equivalently, for any embedding $g:S^{4k-1}\to S^{6k}$ the composition with the
reflection-symmetry of $S^{4k-1}$ represents a knot $-[g]\in E^{6k}(S^{4k-1})$.

(b) We have $r_-\psi_+=r_-$, $r_+\psi_+=-r_+$, $\lambda_+\psi_+=-\lambda_+$, and $H\lambda_-\psi_+=H\lambda_-$.
\end{Lemma}

\begin{proof} (a) This follows by definition of the Haefliger isomorphism $E^{6k}(S^{4k-1})\to\Z$
\cite[\S2]{Ha62l}, \cite[\S3]{Sk16s}.

(b) The equation $r_-\psi_+=r_-$ is clear.
The equation $r_+\psi_+=-r_+$ holds by (a).
We have $\lambda_+\psi_+=\lambda_+\circ(-\iota_{4k-1})=-\lambda_+$.
We have $H\lambda_-\psi_+=H((-\iota_{2k})\circ\lambda_-)=H\lambda_-$ by Lemma \ref{l:circ}.
\end{proof}

In the rest of this section denote $\pi:=\pi_{4k-1}(S^{2k})$.

\begin{Remark}\label{r:chor}
For every $k>1$ there is an embedding $f:T^{0,4k-1}\to\R^{6k}$ whose restriction to each component is isotopic to the standard embedding, but which is not isotopic to the embedding obtained from $f$ by changing orientations of both components (i.e. to $\psi_+\psi_-f = \psi_-\psi_+f$).

This follows by taking $[f]=\zeta x$ for any $x\in\pi$ such that $Hx\ne0$ (e.g. $x=[\iota_k,\iota_k]$).
By Symmetry Lemma \ref{l:symkno}.b $(H\lambda_+,H\lambda_-)\psi_+=(-H\lambda_+,H\lambda_-)$.
Analogously $(H\lambda_+,H\lambda_-)\psi_-=(-H\lambda_+,-H\lambda_-)$.
Hence $(H\lambda_+,H\lambda_-)\psi_+\psi_- = (-H\lambda_+,-H\lambda_-)$.
By $\lambda_+\zeta=\id$ and Lemma \ref{l:sym} we have $H\lambda_+\zeta = H\lambda_-\zeta = H$.
So $[f]=\zeta x$ indeed works as an example.

It would be interesting to know if \emph{there is an ordered oriented link $f$ in 3-space whose restriction to each component is the unknot, but which is not isotopic to the link obtained from $f$ by changing orientations of both components}.
I am grateful to S. Chmutov and V. Mantourov for informing me that such an example is unknown.
\end{Remark}

\begin{proof}[Proof of Theorem \ref{p:sigma}.b]
Theorem \ref{p:sigma}.b follows because on $\ker r_+$ we have
$$2r_- \overset{(1)}= 2\#\psi_+\sigma \overset{(2)}= (2r_- \pm H\lambda_+ \pm H\lambda_-)\psi_+\sigma \overset{(3)}= (2r_- \mp H\lambda_+ \pm H\lambda_-)\sigma \overset{(4)} =$$
$$= 2\# \mp H\lambda_+\sigma \pm H\lambda_- \overset{(5)} =
2r_- \pm H\lambda_+ \pm H\lambda_- \mp H\lambda_+\sigma \pm H\lambda_-,\quad\text{where}$$

$\bullet$ equality (1) holds because two copies of the first component having opposite orientations `cancel';

$\bullet$ equality (3) holds by Symmetry Lemma \ref{l:symkno}.b;

$\bullet$ equality (4) holds by Theorem \ref{p:sigma}.a;

$\bullet$ equality (5) holds by Connected Sum Theorem \ref{l:cosu} because $r_+=0$ on $\ker r_+$.

Let us prove equality (2).
On $\ker r_+$ we have $r_+=r_+\sigma=0$.
Change of the orientation of the \emph{standard} embedding $S^q\to\R^m$ gives embedding $S^q\to\R^m$ isotopic to the standard one.
Hence $r_+\psi_+\sigma=0$.
Now equality (2) holds by Connected Sum Theorem \ref{l:cosu}.
\end{proof}

\aronly{
\begin{proof}[Sketch of an alternative proof of Theorem \ref{p:sigma}.b for $k=1$]
Take a representative $f:T^{0,4k-1}\to S^{6k}$ of an element from $\ker r_+$.
Take a representative $g$ of $\sigma[f]$.
Analogously to \cite[\S4]{Wa66} there is a unique
framing of $f$ such that $p_k(M_f)=0$ for the $6k$-manifold $M_f$ obtained from $S^{6k}$ by surgery along $f$ with this framing.
Denote by $f_\pm\in H_{4k}(M_f)$ `the homology classes of handles'.
Analogously to \cite[Theorem 4]{Wa66}, \cite{Sk06'}
$H\lambda_\pm[f]=f_\pm f_\mp^2$ and $6r_\pm[f]=f_\pm^3$.
There is `sliding handles' diffeomorphism $M_f\to M_g$.
Under this diffeomorphism $g_+,g_-$ go to $f_+,f_++f_-$.
Since $r_+[f]=0$, we obtain the required relations.\footnote{Sketches of alternative proofs of Theorems \ref{l:cosu} and \ref{p:sigma}.b would work for any $k$ if one proves higher-dimensional analogue of \cite[\S4]{Wa66}.}
\end{proof}
}

\begin{Remark}\label{r:doubling}
Let $D:\ker r_+\to E^m(S^q_1\sqcup S^q_2\sqcup S^q_3)$ be the `doubling' of the first component.
The map $D$ is well-defined and is a homomorphism \cite[Lemmas 3.1--3.3]{Sk11}.
Clearly,
$$\lambda_{12}D=\lambda_{21}D=0,\quad \lambda_{13}D=\lambda_{23}D=\lambda_+,\quad\text{and}\quad\lambda_{31}D=\lambda_{32}D=\lambda_-.$$
Clearly, $\sigma=\#_{23}D$.
Hence for $m=6k$, $q=4k-1$ and $k>1$ by Theorems \ref{t:three}.c and \ref{p:sigma}.b we have
$$H\lambda_++2H\lambda_- = H\lambda_+\sigma = H\lambda_+\#_{23}D =
\pm2\mu D+H(\lambda_{12}+\lambda_{13})D = \pm2\mu D+H\lambda_+.$$
Hence $\mu D=\pm H\lambda_-$.
\end{Remark}

\begin{proof}[Proof of Corollary \ref{t:nonis}]
We have $\lambda_+s_\pm =0$.
So for the first statement it suffices to prove that $H\lambda_+(\sigma\zeta-\zeta) w\ne0$.
This follows because
$$H\lambda_+\zeta w \overset{(1)}= Hw = 2\ne6 = 3Hw \overset{(4)}= (H\lambda_++2H\lambda_-)\zeta w \overset{(5)}= H\lambda_+\sigma\zeta w,\quad\text{where}$$

$\bullet$ (1) follows because $\lambda_+\zeta=\id\pi$;


$\bullet$ (4) follows because $\lambda_+\zeta=\id\pi$ and by Lemma \ref{l:sym};

$\bullet$ (5) follows by Theorem \ref{p:sigma}.b.

The statement `in other words' follows from the first statement by \cite[Theorem 2.4]{Ha66C}, see also \cite[Theorem 7.1]{Sk16h}.
\end{proof}


In the rest of this section assume that $m=6k$, $q=4k-1$, and denote (see the text after Connected Sum Theorem \ref{l:cosu})
$$h_{+,-}:=\frac12H(\lambda_++\lambda_-),\quad E=E^m(S^q)\cong\Z.$$
Denote by $\tau:\pi\to E^m(D^1\times S^q)$ the map essentially constructed in the construction of the Zeeman map $\zeta$, so that $r\tau=\zeta$.

\begin{Lemma}\label{l:sigma}  (a) On $K$ we have $ h_{+,-}\sigma=3h_{+,-}$.

(b) We have $h_{+,-}\psi_-\zeta=0$.

(c) We have $h_{+,-}\zeta=H$.

(d) The sequence $\pi \overset\tau \to E^m(D^1\times S^q) \overset{r_+r}\to E$ is exact.
\end{Lemma}

\begin{proof} (a) On $K$ we have $h_{+,-}\sigma = H\lambda_-+\frac12H\lambda_- = 3h_{+,-}$ by Theorem \ref{p:sigma}.

(b) We have
$$h_{+,-}\psi_-\zeta \overset{(1)}= \frac12H(\lambda_+-\lambda_-)\zeta \overset{(2)}=
H-H=0,\quad\text{where}$$

$\bullet$ (1) holds because $\lambda_-\psi_-=-\lambda_-$ and $H\lambda_+\psi_-=H\lambda_+$
by Symmetry Lemma \ref{l:symkno}.b;

$\bullet$ (2) holds because $\lambda_+\zeta=\id\pi$ and by Lemma \ref{l:sym}.

(c) We have $h_{+,-}\zeta = \frac12(H+H)=H$ because $\lambda_+\zeta=\id\pi$ and by Lemma \ref{l:sym}.

(d) This holds by \cite[Theorem 2.5]{Sk11}, \cite[Theorem 1.7]{Sk15}.
\end{proof}

\begin{proof}[Proof of Corollary \ref{t:tworesb}]
Consider the following diagram:
$$E^m(D^1\times S^q)\oplus K\oplus E \overset{\psi_-r\oplus\sigma|_K\oplus s_-}\to E^m(T^{0,q}) \overset{r_+\oplus h_{+,-}}\to E\oplus\Z.$$
The map $r_+\oplus h_{+,-}$ is surjective by Propositions \ref{r:hae}.a,b.
We have $r_+s_-=0$ and $\lambda_+s_-=\lambda_-s_-=0$, so $h_{+,-}s_-=0$.
Hence it suffices to prove that $\theta:=(r_+\oplus h_{+,-})(\psi_-r\oplus\sigma|_K)$ is not surjective.

We have $r_+\psi_-r\tau=r_+\sigma=0$.
So by Lemma \ref{l:sigma}.a,b $\theta(\im\tau\oplus K)\subset0\oplus3\Z$.
Since $r_+\psi_-r=r_+r$, by Lemma \ref{l:sigma}.d we can apply the following simple result to
$A=E^m(D^1\times S^q)\oplus K$, $\alpha=h_{+,-}(\psi_-r\oplus\sigma|_K)$, and $G=\im\tau\oplus K$, so that
$A/G$ is a subgroup of $E\cong\Z$.

\emph{If $G$ is a subgroup of an abelian group $A$, homomorphism $\alpha:G\to\Z$ is not surjective, and $A/G$
is a subgroup of $\Z$, then no extension $A\to\Z\oplus\Z$ of
$0\oplus\alpha$ is surjective.}

We obtain that $\theta$ is not surjective
\end{proof}

\begin{Lemma}\label{l:sigmala} We have $\Sigma\lambda_+\sigma=\Sigma\lambda_+$.
\end{Lemma}

\begin{proof} We have $\Sigma\lambda_+\sigma=-\Sigma\lambda_-\sigma=-\Sigma\lambda_-=\Sigma\lambda_+$
by Theorem \ref{p:sigma}.a and the Haefliger Theorem \ref{t:hae}.
\end{proof}

\begin{proof}[Proof of Corollary \ref{t:twores}]
\aronly{\footnote{
The notation $r$ of this paper agrees with \cite{Sk15} and so with earlier papers, but disagrees with arXiv version 1 of this paper.
The latter version has some confusion between $r$ and $\psi_-r$.}}
Consider the following diagram:
$$\xymatrix{
E^m(D^1\times S^q)\oplus K\oplus E \ar[rr]^(.6){r\oplus\sigma|_K\oplus s_-} \ar[d]_{\rho\oplus h_{+,-}\oplus\id E} & & E^m(T^{0,q}) \ar[d]^{\gamma:=r_+\oplus\lambda_+\oplus h_{+,-}\oplus r_-}\\
(E\oplus\pi)\oplus\Z\oplus E \ar[rr]_{\varphi} & & E\oplus\pi\oplus\Z\oplus E
},\qquad\text{where}$$
$$\rho:=r_+r\oplus\lambda_+r\quad\text{and}\quad
\varphi(x,y,z,t)\ :=\ (x,\ y+2zw+\theta_1x,\ 3z+Hy+\theta_2x,\ t+x\pm z)$$
for some homomorphisms $\theta_1:E\to\pi$ and $\theta_2:E\to\Z$.
Here $\pm$ is the sign depending only on $k$, the same as in Connected Sum Theorem \ref{l:cosu}.


By Proposition \ref{r:hae}.c $\gamma$ is injective.
In the following paragraph we prove that
$\gamma$ is surjective.

Denote by $\overline\omega$ the link obtained from $\omega$ by the exchange of the components.
Define
$$\gamma':E\oplus\pi\oplus\Z\oplus E\to E^m(T^{0,q})\quad\text{by}\quad
\gamma'(a,b,c,d):=s_+a+\zeta b+(c-Hb)\overline\omega+s_-d.$$
The map $\gamma$ is surjective because
$$\gamma\circ\gamma' = r_+s_+\oplus \lambda_+\zeta\oplus h_{+,-}\gamma' \oplus  r_-s_- =
\id E\oplus\id\pi\oplus\id\Z\oplus\id E\quad\text{because}$$
$$r_\pm\zeta=r_\pm \omega=r_+s_-=r_-s_+=0,\quad \lambda_+s_+ = \lambda_+s_- = \lambda_+\overline\omega = 0,\quad
h_{+,-}s_+ = h_{+,-}s_- = 0,$$
$$\text{and}\quad h_{+,-}(\zeta b+(c-Hb)\overline\omega) = Hb+(c-Hb) = c.$$

By Proposition \ref{r:hae}.c,d $h_{+,-}$ is an isomorphism.
Since $r\tau=\zeta$, we have $\lambda_+r\tau=\id\pi$ and $r_+r\tau=0$.
So by Lemma \ref{l:sigma}.d $\rho$ is an isomorphism.
(Cf. the case $p=0$ of \cite[Remark 1.8.a]{Sk15}.)
Hence the left vertical map $\rho\oplus h_{+,-}\oplus\id E$ is an isomorphism.


In the following paragraph we prove that $\varphi$ is an isomorphism.

This $\varphi$ is a linear map whose matrix is
$$\left(\begin{matrix} 1 & \theta_1 & \theta_2 & 1 \\ 0& \id\pi & H & 0 \\ 0& 2w & 3 & \pm1 \\ 0& 0 & 0 & 1 \end{matrix}\right).$$
So it suffices to prove that the self-map
$\varphi'=\left(\begin{matrix}\id\pi& H \\ 2w & 3 \end{matrix}\right)$ of $\pi\oplus\Z$ is an isomorphism
By Theorem \ref{t:whit} the group $\pi$ is the sum of $\Z$ and a finite group.
The map $\varphi'$ maps the torsion subgroup of $\pi$ to itself isomorphically.
The determinant of $\varphi'$ on the free part is $3-2Hw=-1$.
Hence $\varphi'$ is an isomorphism.

Thus it suffices to prove that the diagram is commutative (for some $\theta_1,\theta_2$).

Take any $t\in E$.
Both compositions of the diagram map $(0,0,t)$ to $(0,0,0,t)$ because $r_+s_-=\lambda_\pm s_-=0$, and $r_-s_-t=t$.

In this paragraph we prove that the diagram is commutative on $K$.
Take any $u\in K$.
Denote $z:=h_{+,-}u$.
Then both compositions of the diagram map $(0,u,0)$ to $\varphi(0,0,z,0)=(0,2zw,3z,\pm z)$ because

$\bullet$ $r_+\sigma=0$;

$\bullet$ on $K$ we have $\Sigma\lambda_+\sigma = \Sigma\lambda_+ = 0$ by Lemma \ref{l:sigmala}, so by Theorems \ref{p:sigma}.b and \ref{t:whit}
$$\lambda_+\sigma u = \frac12(H\lambda_+\sigma u)w = (H\lambda_-u)w = (2 h_{+,-}u)w = 2zw;$$

$\bullet$ $h_{+,-}\sigma=3 h_{+,-}$ on $K$ by Lemma \ref{l:sigma}.a; and

$\bullet$ $r_-\sigma=\#=\pm h_{+,-}$ on $K$ by Theorem \ref{p:sigma}.a and Connected Sum Theorem \ref{l:cosu}.

In the following two paragraphs we prove that the diagram is commutative on
\linebreak
$E^m(D^1\times S^q)$.
Any element of $E^m(D^1\times S^q)$ equals $\rho^{-1}(x,y)$ for some $x\in E$ and $y\in\pi$.

Both compositions of the diagram map $(\rho^{-1}(0,y),0,0)$ to $\varphi(0,y,0,0)=(0,y,Hy,0)$ because
$r\tau=\zeta$ and
$$r_+\zeta = r_-\zeta=0,\quad \lambda_+\zeta  = \id\pi\quad\text{and}\quad h_{+,-}\zeta = H$$
(the latter equality holds by Lemma \ref{l:sigma}.c).

For any embedding $f:D^1\times S^q\to\R^m$ the restrictions $f|_{1\times S^q}$ and $f|_{-1\times S^q}$ are isotopic.
So $r_+r=r_-r$.
Define $\theta_1(x) := \lambda_+r\rho^{-1}(x,0)$ and $\theta_2(x):= h_{+,-}r\rho^{-1}(x,0)$.
Then both compositions of the diagram map $(\rho^{-1}(x,0),0,0)$ to $\varphi(x,0,0,0)=(x,\theta_1x,\theta_2x,x)$.
\end{proof}

\aronly{
\begin{Remark}[low-dimensional versions of $\#$ and $\sigma$]\label{r:nonis}
The definition of $\zeta$ after Connected Sum Theorem \ref{l:cosu} works for $m=q+2=3$ and gives the `standard link $\zeta(u):S^1\sqcup S^1\to S^3$ of linking number $u$'.
Let us give a less formal repetition of that construction for $m=q+2=3$.
Let $\zeta(0)$ is  be the trivial link.
Now assume that $u\ne0$.
The first component of $\zeta(u)$ is the standard $S^1=\inc_{3,1}(0\times S^1)$ in $S^3$.
The second component of $\zeta(u)$ is contained in $\inc_{3,1}(\partial D^2\times S^1)$, makes $|u|$ turns around $S^1\times 0$ and is oriented `parallel' to $\inc_{3,1}(0\times S^1)$ when $u>0$ and `opposite' to $\inc|_{S^1\times 0}$ when $u<0$.


Define the link $\#\zeta(u)$ as in the definition of $\#$, taking for connected summation a band close to `standardly twisted' rectangle.
The isotopy class of $\#\zeta(u)$ is independent of the choice of such a band (as opposed to PL band).
It would be interesting to know which knot is $\#\zeta(u)$ (depending on $u$).

Define the link $\sigma\zeta(u)$ as in the definition of $\sigma$, taking for connected summation a band  close to `standardly twisted' rectangle.
The isotopy class of $\sigma\zeta(u)$ is independent of the choice of such a band (as opposed to PL band).
I conjecture that $\sigma\zeta(u)$ is not isotopic to $\zeta(u)$, at least for $|u|>1$.
Perhaps this conjecture could be (dis)proved using calculation of some invariant $\sigma$ \cite[Theorems 3 and 4.1]{Me20}.
\end{Remark}
}


\end{document}